\documentclass[12pt, reqno]{amsart}
\usepackage{amsmath, amsthm, amscd, amsfonts, amssymb, graphicx, color}
\usepackage{setspace}
\usepackage{mathrsfs}

\newtheorem{theorem}{Theorem}[section]
\newtheorem{lemma}[theorem]{Lemma}
\newtheorem{proposition}[theorem]{Proposition}
\newtheorem{corollary}[theorem]{Corollary}
\theoremstyle{definition}

\theoremstyle{remark}

\numberwithin{equation}{section}

\newcommand{\NN}{\mathbb{N}}

\newcommand{\CC}{\mathbb {C}}

\begin{document}
\setcounter{page}{1}
\title[Dynamics of integral and   differentiation  operators ]{Dynamics of the Volterra-type integral and   differentiation  operators on  generalized Fock spaces}

\author[Jos\'e Bonet*]{Jos\'e Bonet* }\footnote{*Corresponding author}
\address{Instituto Universitario de Matem\'atica Pura y Aplicada IUMPA \\
Universitat Polit\`ecnica de Val\`encia\\
E-46071 Valencia, Spain}
\email{jbonet@mat.upv.es}
\thanks{The research of the first author  was partially  supported by the research projects MTM2016-76647-P and  GV Prometeo 2017/102 (Spain).}

\author[Tesfa  Mengestie]{Tesfa  Mengestie }
\address{Department of Mathematical Sciences \\
Western Norway University of Applied Sciences\\
Klingenbergvegen 8, N-5414 Stord, Norway}
\email{Tesfa.Mengestie@hvl.no}

\author [Mafuz Worku]{ Mafuz Worku}
\address{Department of Mathematics,
Addis Ababa University, Ethiopia}
\email{mafuzhumer@gmail.com}
\thanks{The research of the third author  is  supported by ISP project, Addis Ababa University, Ethiopia}


\begin{abstract}
Various dynamical properties of the differentiation and Volterra-type integral operators on generalized Fock spaces are studied. We show that  the differentiation operator is always supercyclic on these spaces. We further characterize when it is  hypercyclic, power bounded and uniformly mean ergodic. We prove that the operator  satisfies  the  Ritt's resolvent condition  if and only if it is power bounded and uniformly mean ergodic.  Some similar results  are obtained for the  Volterra-type   and Hardy integral  operators.
\end{abstract}

\maketitle

\section{Introduction}

For holomorphic functions  $f$ and $g$, the  differentiation operator $Df= f'$ and the Volterra-type integral operator  $V_g f(z)= \int_{0}^z g'(w) f(w)dw$ are classical objects in operator theory, function spaces and differential equations. Many of their  basic  properties   including boundedness, compactness and spectra  have been extensively  studied when acting on several function spaces over various domains; see for example   \cite{AT, Maria, Maria2, Bonet, JBB, BT, Olivia, Diff, TM5, TM6, TM3} and the references  therein. Understanding the dynamical structures of  these operators is another  important  and basic problem in  operator theory.  The main purpose of this paper is to study  such structures  on    generalized  Fock   spaces  $\mathcal{F}^p_{(\alpha, m)}$. We are especially interested in identifying  their  various forms of cyclicity, power boundedness, and uniform mean ergodic properties.

   Let us first set the generalized Fock spaces.   For  $m, \alpha >0$ and $0<p< \infty$,  the  spaces  $\mathcal{F}_{(\alpha, m)}^p$    consist of  all entire functions $f$ for which
\begin{align*}
\|f\|_{(p,\alpha,m)}^p= \int_{\CC} |f(z)|^pe^{-p\alpha |z|^m} dA(z) <\infty
\end{align*}
where   $dA$ denotes the usual Lebesgue area  measure on $\CC$.
Note that  the norm  can be rewritten in terms of the  integral mean $M_p^p(f,r)=\int_0^{2\pi }|f(re^{it})|^p\frac{dt}{2\pi }$ as
\begin{align*}
\|f\|_{(p, \alpha, m)}^p= 2\pi \int_0^{\infty }M_p^p(f,r)re^{-p\alpha r^m}dr.
\end{align*}
We now introduce a few relevant notations needed in the rest of the article. The expression
 $U(z)\lesssim V(z)$ (or
equivalently $V(z)\gtrsim U(z)$) means that there is a constant
$C$ such that $U(z)\leq CV(z)$ holds for all $z$ in the set in
question. We write $U(z)\simeq V(z)$ if both $U(z)\lesssim V(z)$
and $V(z)\lesssim U(z)$.

The norms of the  monomials will  play an important role in studying the dynamical properties of both the differentiation and integral operators on the spaces $\mathcal{F}_{(\alpha, m)}^p$.   Thus, we  estimate them  using   Stirling's formulas,
\begin{align}
\label{factorial}
n! \simeq \sqrt{ n}n^{n}e^{-n} \ \ \text{and} \ \ \Gamma (x+1)\simeq \sqrt{ x}x^x e^{-x}, \ x>0,
\end{align}  where $\Gamma$ denotes the Gamma function,  and we get
\begin{align} \label{estmono}
\|z^n\|_{(p, \alpha, m)}=\bigg(2\pi \int_0^{\infty }r^{pn}e^{-p\alpha r^m}rdr\bigg)^{1/p}=
\frac{\bigg(2\pi\Gamma \Big(\frac{pn}{m}+\frac{2}{m}\Big)\bigg)^{1/p}}{m(p\alpha )^{\frac{n}{m}+\frac{2}{mp}} }\quad \quad \nonumber \\
\simeq \frac{(pn+2-m)^{\frac{n}{m}+\frac{2}{mp}-\frac{1}{2p}}}{(p\alpha )^{\frac{n}{m}+\frac{2}{mp}}m^{\frac{n}{m}+\frac{2}{mp}+\frac{1}{2p}} e^{\frac{n}{m}+\frac{2}{mp}-\frac{1}{p}}}
\simeq \bigg(\frac{n}{me\alpha }\bigg)^{\frac{n}{m}+\frac{2}{mp}-\frac{1}{2p}}. \quad \quad \quad \quad
\end{align}
From the preceding  estimate we have in particular,  for all $n\in \NN$,
 \begin{align}
 \label{lm1}
\|z^n\|_{(p, \alpha,1)}\simeq n! \alpha^{-n}n^{\frac{3-p}{2p}}.
 \end{align}
Next,  we recall some definitions related to iterates of an operator.  Given a Banach space $X$, we denote  by $\mathcal{L}(X)$ the space of continuous linear operators $T$ on $X$.
An operator $T\in \mathcal{L}(X)$ is said to be  hypercyclic if there exists a vector  $x$ in $ X$  such that its orbit, $\{ T^nx; n\in \NN_0=\{0 \}\cup \NN\}, $ is dense in $X$. The operator is called cyclic if the linear span of  an orbit is dense in $X$, and supercyclic whenever a  projective orbit, $\{ \lambda T^nx; n\in \NN_0 , \lambda \in \CC\}$,  is dense in $X$. Obviously,  hypercyclicity is a stronger property than supercyclicity  which in turn  is stronger than  cyclicity. Good references on this subject are \cite{BM,Grosse}.

An operator $T\in \mathcal{L}(X)$ is said to be power bounded if there exists a positive number $M$ such that  $\|T^n\|\leq M $ for all $n\in\NN_0$. The operator $T$ is called quasi-nilpotent if $\lim_{n \rightarrow \infty} ||T^n||^{1/n} = 0$.  It  is said to be mean ergodic if there exists an operator $P\in \mathcal{L}(X)$ such that $$ Px:=\lim_{n\to \infty }\frac{1}{n}\sum_{k=1}^nT^kx, \ x\in X $$ exists in $X$. If the convergence is in the operator norm, then $T$ is called uniformly mean ergodic. The standard references about mean ergodic operators are the books of Krengel \cite{Kr} and Yosida \cite{Y}.

\section{ Dynamics of the differentiation  operator on  $\mathcal{F}^p_{(\alpha, m)}$}

The  differentiation operator $D$ has been studied on Banach spaces of analytic functions by several authors.  Harutyunyan and Lusky \cite{Harutyunyan} identified conditions under which the operator becomes bounded when acting between weighted spaces of holomorphic functions endowed with the supremum norm; see also \cite{AT}.  Bonet \cite{Bonet} studied various dynamical properties of the operator on these weighted  spaces, and the  study was  continued jointly with Beltr\'an, Bonilla and Fern\'andez in \cite{Maria2,JBB}. Later in 2014,  Beltr\'{a}n \cite{Maria} studied the dynamics of the operator on  a wider class of generalized weighted Bergman spaces.    On the other hand, the operator is also  known to act in  unbounded way in several functional spaces. For instance in \cite{UK} Ueki showed its unboundedness on the classical growth type Fock spaces.  Mengestie and Ueki \cite{TM3} verified its  unboundedness on all classical  Fock spaces  and  generalized Fock  spaces where the weight function grows faster than the Gaussian weight function $|z|^2/2$.  The same conclusion was later   drawn in \cite{TM5} on the Fock--Sobolev spaces which are typical  examples of  generalized  Fock spaces with  weight function growing slower than the Gaussian function. Inspired by all these, Mengestie \cite{Diff}  asked  the question of  how fast should the weight function need to grow in order that the corresponding  generalized Fock spaces support a continuous differentiation operator. He  further  considered the spaces $\mathcal{F}_{(1, m)}^p$ and  showed that the weight function should actually  grow much slower than the classical Gaussian function.  More specifically,  it was proved that  the  operator $D$ is bounded on $\mathcal{F}_{(1, m)}^p$, $0<p<\infty $, if and only if  $m\leq 1 $, and compact if and only if $m<1$.  See also \cite[Section 5]{Maria}.   In this section, we  continue those lines of research and investigate the dynamical behaviour of the operator $D$ on  $\mathcal{F}_{(\alpha, m)}^p$.

  \begin{proposition}\label{prop1}
Let $1\leq p<\infty $ and let the differentiation operator $D$  be bounded on $\mathcal{F}^p_{(\alpha, m)}$.   Then
\begin{enumerate}
\item $D$  is hypercyclic on $\mathcal{F}^p_{(\alpha, m)}$ if and only if  either
 $m=1$ and $\alpha >1$ or $m=1$, $\alpha =1$ and $p>3$.
\item $D$  is supercyclic and hence cyclic on $\mathcal{F}_{(\alpha, m)}^p$.
\end{enumerate}
\end{proposition}
\begin{proof} (i) First note that since $D$ is bounded, $m\leq 1$, as can be seen from (\ref{estmono}). In addition,  since no compact operator is hypercyclic on a non zero complex Banach space \cite[Corollary 1.22]{BM}, it follows that $D$ is not hypercyclic on $\mathcal{F}_{(\alpha, m)}^p$  whenever  $m<1$. On the other hand, for $m=1$, using the relation in   \eqref{lm1}, we have
\begin{align*}
\liminf_{n\to \infty }\frac{\|z^n\|_{(p, \alpha, 1)}}{n!}\simeq \liminf_{n\to \infty } \frac{n^{\frac{3}{2p}}}{\alpha^n\sqrt{n}}
=\begin{cases} 0,\ \text{for} \  \alpha =1 \ \text{and} \ p>3 \ \text{or} \ \alpha >1
,\\
1, \ \ \text{for} \ \alpha =1 \ \text{and} \ p=3, \\
\infty ,\ \text{for} \  \alpha <1 \ \text{or} \ \alpha =1 \ \text{and} \ p< 3.
\end{cases}
\end{align*}
Then,  by Theorem 5.2 of \cite{Maria}, $D$ is hypercyclic if and only if either $ \alpha =1 \ \text{and} \ p>3 \ \text{or} \ \alpha >1$. \\
(ii) For this part, we follow the arguments used in the  proof of  \cite[Proposition~2.7]{Bonet}. Since for each $n\in \NN$ the monomial  $z^n$ belongs to the kernel $\text{Ker}  D^{n+1}$ of $D^{n+1}$,  the generalized kernel set  \begin{align*}
G_{\text{Ker}}:=\bigcup_{n=0}^{\infty }\text{Ker} D^n
\end{align*} contains  all the polynomials.  Since the polynomials are dense in $\mathcal{F}_{(\alpha, m)}^p$, it follows that  $G_{\text{Ker}}$ is dense in $\mathcal{F}_{(\alpha, m)}^p$. Moreover, the range of the operator $D$ contains polynomials and therefore it is dense in  $\mathcal{F}^p_{(\alpha, m)}$. Then our conclusion follows after an application of \cite[Corollary~3.3]{Alfred}.
\end{proof}

We note that  the proof of  part (i) depends on the hypercyclicity criterion  due to  B\'es and Peris \cite{BP}, where the original idea goes  back to the work of  Kitai in her Ph.D. thesis  \cite[Theorem 3.4]{Grosse}.  The aforementioned  Theorem 5.2 of  \cite{Maria} ensures that $D$ is hypercyclic  on generalized Bergman spaces  if and only if it satisfies  the hypercyclicity criterion. This was further shown to be equivalent to  a condition like
\begin{align*}
\liminf_{n\to \infty }(n!)^{-1}\|z^n\|_{(p, \alpha, 1)}=0
\end{align*} which remains valid in our setting. Similarly, the proof of part (ii) was based on a density condition in  \cite{Alfred}. This condition is equivalent to the known  supercyclicity criterion; see  \cite[Lemma~3.1]{Alfred}.  Therefore, $D$ satisfies the supercyclicity criterion   if and only it is supercyclic.  We note that not all supercyclic operators satisfy this criterion; see \cite{ DLR} for an  example.

Having completely identified conditions under which $D$ is hypercyclic, we  next consider the question of when $D$ can be topologically  mixing  on  $\mathcal{F}^p_{(\alpha, m)}$.
Recall that  an operator $T$ on a Banach space $X$ is topologically mixing  if for every pair of non-empty open subsets $U$ and $V$ of $X$, there exists an $N\in \NN$ such that $ T^n(U)\cap V \neq \emptyset$ for all $n\geq N$. Note that topologically mixing is a stronger  operator theoretic condition than hypercyclicity in general.  Following  the discussions above and the arguments used back in the proof of  Theorem 2.4 and Corollary~2.6  of \cite{Bonet},  the differentiation operator $ D$ is topologically mixing  on  $\mathcal{F}^p_{(\alpha, m)}$ whenever  it is hypercyclic.

Now we investigate when the differentiation operator is power bounded and uniformly mean ergodic.

 \begin{theorem}\label{thm1}
Let $1\leq p<\infty $ and  the differentiation operator $D$  be bounded on  $\mathcal{F}^p_{(\alpha, m)}$. Then the  following statements are equivalent.
\begin{enumerate}
 \item $D$  is power bounded and uniformly mean ergodic on  $\mathcal{F}^p_{(\alpha, m)}$.
  \item  Either $m<1$ or $m=1$ and $\alpha <1$.
  \end{enumerate}
\end{theorem}
\begin{proof}
We first show that $(ii)$ implies $(i)$.  Since $D$ is bounded, we have $m\leq 1$. If $m<1$, then the operator is compact and by  \cite[Theorem 1.2]{Diff}, its spectrum  $\sigma (D)$  contains only the zero element. By the spectral formula, there exist $\delta <1$ and $N \in \mathbb{N}$ such that
\begin{align}
\label{forresolvent}
\|D^n\|\leq \delta^n \ \ \text{for all} \ n \geq N,
\end{align} and therefore, $D$ is power bounded and uniformly mean ergodic in this case.  For the case when $m=1$ and $\alpha <1$, arguing  as in  (5.3) in  \cite[Proposition~5.9]{Maria}, we get
\begin{align*}
\|D^kf\|_{(p, \alpha, 1)}^p \leq \frac{ (k!)^pe^{p\alpha k}}{k^{pk}}\|f\|_{(p, \alpha, 1)}^p.
\end{align*}
Using this along with Stirling formula \eqref{factorial}, we deduce
\begin{align} \label{estdiff}
\|D^k\|\lesssim \frac{k^{\frac{p}{2}}}{e^{(1-\alpha )k}} \to 0
\end{align}  as $k\to \infty $ for $\alpha <1$, and hence the operator is power bounded in this case as well.
Now,  applying  the  estimate  in \eqref{estdiff},
\begin{align*}
\bigg\| \frac{1}{n}\sum_{k=1}^n D^k\bigg\| \leq \frac{1}{n}\sum_{k=1}^n\|D^k\|
\lesssim \frac{1}{n}\sum_{k=1}^n \frac{k^{\frac{p}{2}}}{e^{(1-\alpha )k}} \to 0
\end{align*}as $n \to \infty$, from which it follows that $D$ is uniformly mean ergodic.

Now we show that (i) implies (ii). Assume that $m=1$.  Using exponential functions $e_{\beta}(z)=e^{\beta z}, \ |\beta| < \alpha$, we get $\overline{D(0,\alpha )} \subset \sigma (D)$, where $D(0,\alpha )$ is a disc with center $0$ and radius $\alpha$. Thus, $1$ is an accumulation point of $\sigma (D)$  whenever  $\alpha \geq 1$. Since $D$ is assumed to be power bounded, we apply  \cite[Theorem 3.16]{Dunford} (see also \cite[Theorem 2.7]{Kr}), to get that the operator cannot be uniformly mean ergodic in this case. Therefore, we must have  $\alpha < 1$.
\end{proof}

\section{Dynamics of the Volterra-type integral  operator on  $\mathcal{F}^p_{(\alpha, m)}$}

In this section we investigate the dynamics of Volterra-type integral operators
$$V_g f(z)= \int_{0}^z g'(w) f(w)dw.$$
Various aspects of the   operator   which includes boundedness, compactness, and spectra  have  been well   studied  in large class of  function spaces; see for examples   \cite{BT,Olivia2,Olivia,TM5,MW, TM3} and the references therein.  Much less is  known about its  dynamical and mean ergodic properties,  except in the special case when  the symbol $g$ is   the identity map. The  fact that the iterates of the operator involve multiple  integrals makes it difficult to get best possible  estimates of the  norms. In this paper,  we begin the study of the dynamical properties of $V_g$ on the  spaces $\mathcal{F}_{(\alpha, m)}^p$.  It was  shown  \cite{Olivia} (see also \cite{Olivia2}), and in \cite{BT} for $p=\infty$, that $V_g$ is bounded on   $\mathcal{F}_{(\alpha, m)}^p$ if and only if $g$ is a complex polynomial  of degree $l$ not bigger than $m$ ($l \leq m$), and $V_g$  is compact in this space if and only if the degree $l$ of $g$ is strictly smaller than $m$ or $m$ is not a positive integer.

\begin{proposition}
\label{prop2}
Let $1\leq p< \infty$. Let $V_g$ be bounded on $\mathcal{F}^p_{(\alpha, m)}$ and hence
$g(z)= a_lz^l+a_{l-1}z^{l-1}+...+a_1z+a_0, \ l \leq m$.   Then
\begin{enumerate}
\item $V_g$ is   not supercyclic  on $\mathcal{F}^p_{(\alpha, m)}$ and hence not hypercyclic.
\item If $g(z)=az^l+b, a \neq 0,$ then $V_g$ is cyclic if and only if $l=1$.
 \end{enumerate}
\end{proposition}
\begin{proof} (i) Since $V_gf(0)=0$ for every $f$ in $\mathcal{F}^p_{(\alpha, m)}$,  the projective orbit of $f$ under  $V_g$ contains only functions that vanish at zero. Thus, it  cannot be dense in $\mathcal{F}_{(\alpha, m)}^p$,  which implies that $V_g$ is not  supercyclic on  $\mathcal{F}^p_{(\alpha, m)}$ and hence not hypercyclic either.  \\

(ii) If $l=1$, then  $(V_g)^n(\textbf{1})(z)=\frac{a^n z^n}{n!}$. Hence
$$\{(V_g)^n(\textbf{1}), n\geq 0\}=\{1, az, \frac{a^2 z^2}{2},\cdots, \frac{a^n z^n}{n!}, \cdots \}.$$ The linear span of the  latter   is known to be dense in $\mathcal{F}_{(\alpha, m)}^p$.

If $l > 1$ and $f$ belongs to $\mathcal{F}^p_{(\alpha, m)}$, using a functional that annihilates the $(l-1)$-th Taylor polynomial of $f$, together with the monomials $z^k, k \geq l,$ we conclude that $f$ cannot be a cyclic vector for $V_g$.
\end{proof}

\noindent The class of Volterra-type integral operators  include the classical integration operator $Jf(z)= \int_{0}^z f(w)dw$  in particular when $g(z)=z$.  By \cite{Olivia, TM3}, \cite{AT} or \cite[Lemma 1.1]{MW},  it follows that  $J$  is bounded on $\mathcal{F}^p_{(\alpha, m)}$ if and only if $m\geq 1$ and compact if and only if $m> 1$.  These conditions are opposite to the corresponding conditions for the differentiation operator $D$ except  when $m=1$,  in  which case both $J$ and $D$ are bounded. Clearly $DJf= f$ and $JDf(z)= f(z)-f(0)$ for all  $z\in \CC$ and $f$ in $ \mathcal{F}^p_{(\alpha, m)}$.
Observe  that while $D$ is supercyclic, there exists no vector whose projective orbit under  $J$ is dense in $\mathcal{F}^p_{(\alpha, m)}$.

 To state our next main result, we first recall some definitions.  Denote by $\mathcal{H}(\CC)$ the set of entire functions on $\CC$. For  $r\geq 0$ and each  $f\in \mathcal{H}(\CC)$,  set
\begin{align*}
M_\infty (f,r)= \sup_{|z|=r} |f(z)|,
\end{align*} and  define  the growth type space $\mathcal{F}^\infty_{(\alpha, m)} $ as the space of functions $f\in \mathcal{H}(\CC)$ such that
\begin{align*}
 \|f\|_{(\infty, \alpha, m)}= \sup_{r>0} e^{-\alpha r^m} M_\infty(f,r) <\infty.
\end{align*}
The estimate corresponding to \eqref{estmono} becomes
\begin{align}
\label{estinf}
\|z^n\|_{(\infty, \alpha, m)}\simeq \bigg(\frac{n}{me\alpha }\bigg)^{\frac{n}{m}}.
\end{align}
For $m=1$, the space $\mathcal{F}^p_{(\alpha, 1)}$ is denoted by  $B_{p,p}(\alpha)$ in \cite{Maria}.
To simplify the notation below, we write the symbol $g(z)= a_lz^l+a_{l-1}z^{l-1}+...+a_1z+a_0$, with $l \leq m$,  as $g= g_l+ g_{l-1}$ where  $g_l(z)=a_l z^l$ and $g_{l-1}(z)=a_{l-1}z^{l-1}+...+a_1z+a_0$. Then the operators on $\mathcal{F}^\infty_{(\alpha, m)} $ satisfy
  \begin{align*}V_g= V_{g_l}+ V_{g_{l-1}},
  \end{align*}  of which $V_{g_{l-1}}$ is always compact and quasi-nilpotent.

Following \cite{AAB}, for each $\lambda \in \CC$, $a \in \CC$ and $m \in \NN$,  the  operator $K_\lambda$ is defined on  $\mathcal{H}(\CC)$  by
\begin{align*}
K_\lambda f(z)= a m e^{\lambda z^m} \int_{0}^z e^{-\lambda w^m} w^{m-1} f(w) dw
= a m z^m \int_{0}^1 e^{\lambda z^m(1- t^m)} t^{m-1} f(tz) dt
 \end{align*} Thus,  when $\lambda= 0$ $m=l \in \NN$ and $a=a_l$, then $K_\lambda$  is just  the  Volterra-type integral operator $V_{g_l}$.

 \begin{lemma}
\label{lem2} Let $  m\geq 1, m \in \NN$.
\begin{enumerate}
\item If $ \ |\lambda|<\alpha$, then  the operator $K_\lambda$ is continuous on $\mathcal{F}^\infty_{( \alpha, m)}$ with operator norm $\|K_\lambda\|\leq \frac{|a|}{\alpha-|\lambda|} $.

\item If $l=m \in \NN$, then $V_{g_l}$ is continuous on $\mathcal{F}^\infty_{( \alpha, m)}$ and its operator norm satisfies $\|V_{g_l}\| \leq \frac{|a_l|}{\alpha} $.
\end{enumerate}
\end{lemma}
\begin{proof}
(i) For $r>0$, we have
\begin{align*}
e^{-\alpha r^m} M_\infty(K_\lambda f, r) = e^{-\alpha r^m} |a| m r^m \int_{0}^1  e^{|\lambda| r^m(1- t^m)} t^{m-1} M_\infty(f, rt) dt  \\
= |a| m r^m \int_{0}^1 t^{m-1} e^{-\alpha (tr)^m} e^{(\alpha - |\lambda|) r^m(t^m-1)} M_\infty(f, rt) dt  \\
\leq |a| m r^m \|f\|_{(\infty,\alpha,m)} \int_{0}^1 t^{m-1} e^{(\alpha - |\lambda|) r^m(t^m-1)} dt \leq
\frac{|a|}{\alpha-|\lambda|} \|f\|_{(\infty,\alpha,m)}.
\end{align*}
Therefore
$$
\|K_\lambda f\|_{(\infty,\alpha,m)} = \sup_{r>0}e^{-\alpha r^m} M_\infty(K_\lambda f, r) \leq 
\frac{|a|}{\alpha-|\lambda|} \|f\|_{(\infty,\alpha,m)}.
$$

(ii) This is a direct consequence of part (i) for $l=m$, $a= a_l$ and $\lambda = 0$.
\end{proof}

\begin{theorem}\label{thm2}
Let $1\leq p\leq \infty $, $m \geq 1$ and $l \in \NN$. Assume that the operator  $ V_g$ is bounded on $\mathcal{F}^p_{(\alpha, m)} $, with $g(z)= g_l(z)+ g_{l-1}(z), \ g_l(z)=a_l z^l, \ l \leq m$. Then  
\begin{enumerate}
\item If $m>1$, $1\leq p\leq \infty $ and $l < m$, then $ V_g$ is compact, quasi-nilpotent, hence power bounded, and uniformly mean ergodic  on $\mathcal{F}^p_{(\alpha, m)}$.
\item If $l=m \in \mathbb{N}$ and $p=\infty$, then $V_g$ is power bounded if and only if $|a_l| \leq \alpha$.
\item If $l=m \in \mathbb{N}$, $p=\infty$ and $|a_l| \leq \alpha$, then $V_{g_l}$ is uniformly mean ergodic  if and only if $|a_l| < \alpha$.
\item If $l=m \in \mathbb{N}$ and $1\leq p < \infty $ and $V_g$ is power bounded, then $|a_l| \leq \alpha$.
\item If $l=m \in \mathbb{N}$ and $1\leq p < \infty $ and $V_{g_l}$ is power bounded and uniformly mean ergodic, then $|a_l| < \alpha$.
\end{enumerate}
\end{theorem}
\begin{proof}
(i) From  results in  \cite{BT,TM3}  for  $p= \infty$ and in \cite{Olivia} for $1\leq p<\infty$, the operator $V_{g}$ is compact on $\mathcal{F}_{(\alpha, m)}^p $ for all $1\leq p\leq \infty$, since $l < m$. Moreover, by \cite[Theorem]{BJ} and \cite[Theorem~1]{Olivia2},  we have  $\sigma (V_{g})= \{ 0\}$, hence $V_g$ is quasi-nilpotent. By  the spectral radius formula, there exist $\beta<1$ and $N \in \mathbb{N}$ such that
\begin{align}
\label{uni}
\|V_{g}^n\|\leq \beta^n \ \ \text{for all} \ n \geq N.
\end{align}
This shows that $V_{g}$ is power bounded.

The operator $V_{g}$ is  also  uniformly mean ergodic in this case. Indeed, an application of  \eqref{uni}  yields, for some $C>0$ depending on $N$,
\begin{align*}
\frac{1}{n}\bigg \|\sum_{k=1}^nV_{g}^k\bigg\|\leq \frac{1}{n}\sum_{k=1}^n \big\|V_{g}^k\big\|\leq \frac{C}{n} + \frac{\beta}{n(1-\beta)} \to 0 \ \text{as}\  n\to \infty.
\end{align*}

(ii) Now $p=\infty$ and $l=m \in \NN$. Assume that $|a_l| \leq \alpha$. By  Lemma~\ref{lem2} (ii) we have, for each $n \in \NN$,
\begin{align*}
\|V_{g_{l}}^n\| \leq |a_l|^n/\alpha^n\leq 1.
\end{align*}
This implies that $V_{g_{l}}$ is power bounded. Since the sum of a power bounded operator and a quasi-nilpotent operator is power bounded, we conclude that $V_g$ is power bounded.

Conversely, suppose that $V_g$ is power bounded. Since $V_{g_{l-1}}$ is quasi-nilpotent by part (i), we get that $V_{g_{l}}$ is power bounded. However, for each positive integer $n$,  a straightforward  integral computation  gives
  \begin{align*}
  V_{g_{l}}^n(\textbf{1})(z)= \frac{l^n(a_l)^n z^{ln}}{\prod_{j=1}^n (jl)}= \frac{ (a_l)^n z^{ln}}{n!}
  \end{align*} from which and \eqref{estinf} we have
    \begin{align*}
\|V_{g_{l}}^n\|\geq \frac{\|V_{g_l}^n\textbf{1}\|_{(\infty, \alpha, l)}}{\|\textbf{1}\|_{(\infty, \alpha, l)}}
\simeq  \frac{|a_l|^n}{n! (\alpha el)^n}  (nl)^{n}
\simeq \frac{|a_l|^n}{\alpha^n} \frac{1}{\sqrt n} \to \infty \quad \quad \quad \quad  \quad \quad
\end{align*}as $n\to \infty$ whenever $|a_l|>\alpha $, a contradiction. Thus $|a_l| \leq \alpha$.

(iii) In this part we assume $p=\infty$, $l=m \in \NN$ and  $|a_l| \leq \alpha$. We first suppose that $|a_l| < \alpha$. By  Lemma~\ref{lem2} we have, for each $n \in \NN$,
\begin{align*}
\|V_{g_{l}}^n\| \leq |a_l|^n/\alpha^n.
\end{align*}
Hence
\begin{align*}
\frac{1}{n} \bigg\|\sum_{k=1}^n V_{g_{l}}^k\bigg\| \leq \frac{1}{n} \sum_{k=1}^n \big\|V_{g_{l}}^k\big\|\leq  \frac{|a_l|}{n(\alpha-|a_l|)} \to 0  \ \text{as}\ \ n\to \infty,
\end{align*}
and $V_{g_{l}}$ is uniformly mean ergodic.

Conversely, suppose that $V_{g_{l}}$ is uniformly mean ergodic and that $|a_l| \leq \alpha$. Part (ii) implies that $V_{g_{l}}$ is power bounded. From \cite{BJ} it follows that the spectrum  $ \sigma(V_{g_{l}}) =\{\lambda \in \CC:  |\lambda|\leq \frac{|a_l|}{\alpha}\}$. Thus, if $|a_l|= \alpha $, then $1$ is an accumulation point of $\sigma (V_{g_{l}})$. By \cite[Theorem 3.16]{Dunford} (see also \cite[Theorem 2.7]{Kr}), the operator $V_{g_{l}}$ is not uniformly mean ergodic. Therefore $|a_l|< \alpha $.

(iv) If $l=m \in \mathbb{N}$ and $1\leq p < \infty $ and $V_g$ is power bounded, then $V_{g_{l}}$ is power bounded since $V_{g_{l-1}}$ is quasi-nilpotent by part (i). Now,
an integral computation again   gives, for each positive integers $k$ and $n$,
  \begin{align*}
  V_{g_{l}}^n(z^k)= \frac{(a_l)^n l^n z^{ln+k}}{\prod_{j=1}^n (jl+k)}
  \end{align*} and hence
  \begin{align}
  \label{goodest}
  \| V_{g_l}^n\| \gtrsim  \limsup_{k\to \infty}\frac{\|V_{g_l}^n(z^k)\|_{(p,\alpha, l)}}{ \|z^k\|_{(p,\alpha, l)}}=  \limsup_{k\to \infty}\frac{|a_l|^n l^n \|z^{nl+k}\|_{(p,\alpha, l)}}{\prod_{j=1}^n (jl+k) \|z^k\|_{(p,\alpha, l)}}\quad \nonumber \\
  \geq \limsup_{k\to \infty}\frac{|a_l|^n l^n\|z^{nl+k}\|_{(p,\alpha, l)}}{(nl+k)^n  \|z^k\|_{(p,\alpha, l)}}, \quad \quad \quad
  \end{align} where the last inequality follows since
   \begin{align*}
  \prod_{j=1}^n (jl+k)= e^{\sum_{j=1}^n \log (jl+k)} \leq e^{n\log (nl+k)}= (nl+k)^n .
  \end{align*}
  Applying the norm estimate in \eqref{estmono},
  \begin{align*}
  \frac{\|z^{nl+k}\|_{(p,\alpha, l)}}{ \|z^k\|_{(p,\alpha, l)}} \simeq \frac{(nl+k)^n}{(e\alpha)^n} \Big(1+ \frac{nl}{k}\Big)^{\frac{k}{l}+\frac{2}{pl}-\frac{1}{2p}}
  \end{align*}
     and plugging this in \eqref{goodest} and  making further simplifications
    \begin{align*}
    \| V_{g_l}^n\| \geq  \frac{|a_l|^n}{\alpha^n} \limsup_{k\to \infty}
    \frac{\Big(1+ \frac{nl}{k}\Big)^{\frac{k}{l}}}{e^n} \Big(1+ \frac{nl}{k}\Big)^{\frac{2}{pl}-\frac{1}{2p}} \quad \quad \quad \quad \quad \quad \\
    \geq\frac{|a_l|^n}{\alpha^n} \limsup_{k\to \infty}
    \frac{\Big(1+ \frac{nl}{k}\Big)^{\frac{k}{l}}}{e^n} = \frac{|a_l|^n}{\alpha^n},
    \end{align*} which yields $\alpha \geq |a_l|$.

(v) Suppose that $V_{g_{l}}$ is power bounded and uniformly mean ergodic.  It follows from \cite{Olivia2} that the spectrum  $ \sigma(V_{g_{l}}) =\{\lambda \in \CC:  |\lambda|\leq \frac{|a_l|}{\alpha}\}$. Hence, if $|a_l|= \alpha $, then $1$ is an accumulation point of $\sigma (V_{g_{l}})$. By \cite[Theorem 3.16]{Dunford} (see also \cite[Theorem 2.7]{Kr}), the operator $V_{g_{l}}$ is not uniformly mean ergodic. This implies $|a_l|< \alpha $.
\end{proof}

\begin{corollary}\label{cor2}
Let $1\leq p\leq \infty $ and $m \geq 1$. Then the integration operator $J$ on  $\mathcal{F}^p_{(\alpha, m)}$ satisfies
\begin{enumerate}
\item If $m>1$ and $1\leq p\leq \infty $, then $J$ is compact, quasi-nilpotent, hence power bounded, and uniformly mean ergodic  on $\mathcal{F}^p_{(\alpha, m)}$.
\item If $m=1$ and $p=\infty$, then $J$ is power bounded if and only if $\alpha \geq 1$.
\item If $m=1$ and $p=\infty$, then $J$ is uniformly mean ergodic  if and only if $\alpha > 1$.
\item If $m=1$ and $1\leq p < \infty $ and $J$ is power bounded, then $\alpha \geq 1$.
\item If $m=1$ and $1\leq p < \infty $ and $J$ is power bounded and uniformly mean ergodic, then $\alpha >  1$.
\end{enumerate}
\end{corollary}
\begin{proof}
This follows immediately from Theorem \ref{thm2}.
\end{proof}

\section{Dynamics of the Hardy operator on  $\mathcal{F}^p_{(\alpha, m)}$}

 In this section we  study  the dynamics of the   classical Hardy operator
$Hf(z)=\frac{1}{z}\int_0^z f(w)dw= \frac{1}{z} Jf(z)$ on the spaces $\mathcal{F}^p_{(\alpha, m)}$.
Dynamical properties of this operator has been investigated in related contexts in \cite{Maria, Maria2}.

    \begin{theorem}
\label{thm4}
Let $1\leq p< \infty $. Then the Hardy operator $H$ is both power bounded and uniformly mean ergodic on $\mathcal{F}^p_{(\alpha, m)}$.
Furthermore, $\| H\|= 1$.
\end{theorem}
\emph{Proof}. Let us first show that $H$ is bounded and $\| H\|= 1$. Proceeding as in \cite{Maria}, we get
\begin{align*}
M_p^p(Hf,r) \leq M_p^p(f,r).
\end{align*}
Multiplying both sides by $2\pi r e^{-p\alpha r^m}$ and integrating over $r$ yields
\begin{align*}
\|Hf\|_{(p,\alpha, m)}^p \leq \|f\|_{(p,\alpha, m)}^p, \quad \quad \quad\quad \quad \quad
\end{align*} which implies that $H$ is bounded and $\|H\| \leq 1.$

On the other hand, if $H$ is bounded, then
\begin{align*}
\|H\textbf{1}\|_{(p,\alpha, m)}^p= \int_{\CC} |H\textbf{1}(z)|^p e^{-p\alpha |z|^m} dA(z)= \|\textbf{1}\|_{(p,\alpha, m)}^p
\end{align*} and hence $\|H\|=1$.

Next fix  $n>1$. Let $f(z)= \sum_{k=0}^\infty a_k z^k$ be the   Taylor series expansion of  $f\in \mathcal{F}^p_{(\alpha, m)}$. A simple  integral computation
 shows that
 \begin{align}
 \label{simple}
 H^n f(z)=\sum_{k=0}^\infty \frac{a_kz^k}{(k+1)^n}.
 \end{align}
 On the other hand, for each $r>0$, applying Cauchy inequalities
 \begin{align*}
 |a_k|r^k = \frac{r^k}{2\pi}\bigg|\int_{|\zeta|= r} \frac{|f(\zeta)|}{\zeta^{k+1}} d\zeta\bigg| \leq \frac{1}{2\pi}\int_{0}^{2\pi} |f(re^{i\theta)}|d\theta \leq M_p(f,r),
 \end{align*}
from which   we get   $ |a_k| \|z^k\|_{(p,\alpha, m)} \leq \|f\|_{(p,\alpha, m)}$  holds  for $k\geq 0$. This and \eqref{simple} imply
 \begin{align*}
 \|H^n f\|_{(p,\alpha, m)}\leq \sum_{k=0}^\infty \frac{|a_k|\|z^k\|_{(p,\alpha, m)}}{(k+1)^n}  \leq  \| f\|_{(p,\alpha, m)} \sum_{k=0}^\infty \frac{1}{(k+1)^n}\simeq  \frac{\| f\|_{(p,\alpha, m)}}{n-1}
 \end{align*} for all $n>1$.
 Therefore, $H$ is power bounded.  Observe also that,
\begin{align*}
\Big\| \frac{1}{n}\sum_{j=1}^n H^j\Big\| \leq \frac{1}{n}\sum_{j=1}^n\|H^j\|
= \frac{1}{n}\bigg(\|H\|+\sum_{j=2}^n\|H^j\|\bigg)= \frac{1}{n}+ \frac{1}{n} \sum_{j=2}^n\|H^j\|\quad \quad \\
\lesssim  \frac{1}{n}+  \frac{1}{n}\sum_{j=2}^n \frac{1}{j-1} \simeq  \frac{1}{n}+  \frac{\log |n-1|}{n}  \to 0\quad \quad
\end{align*}as $n \to \infty$. Therefore, $H$ is also uniformly mean ergodic, which completes the proof.

We now mention consequences of  Theorem~\ref{thm4}. Since  all orbits $\{T^nf; n=0,1,2,...\}$ of any  power bounded operator $T$ are bounded, it  cannot be hypercyclic.  This  conclusion fails to hold for the supercyclicity property in general. But if the operator satisfies in addition  for example $Tf(\zeta)=f(\zeta)$ for all $f$ in the space and at least one point $\zeta\in \CC$, then  $T$ is not supercyclic either. We will prove this for the operator $H$, and  the same proof works in general for any other power bounded operator  $T$.

  \begin{corollary}
 \label{cor3}
 Let $1\leq p< \infty $. Then the Hardy operator $H$ is not supercyclic on $\mathcal{F}^p_{(\alpha, m)}$.
  \end{corollary}
\begin{proof}
For each $f$ and each $n$ we have
  \begin{align}
  \label{Hopitals}
  Hf(0)= f(0)= H^nf(0).
  \end{align}
Proceeding by contradiction, suppose that $f$ is a supercyclic vector for $H$. Then for the constant function $\textbf{1}$, there exits  a sequence
  $(\lambda_k H^{n_k}f)$ in the projective orbit such that  $\lambda_k H^{n_k}f \to \textbf{1} $ as $k\to \infty$.  Consequently,  by \eqref{Hopitals},
  \begin{align*}
  \lim_{k\to \infty} \lambda_k H^{n_k}f(0)= \lim_{k\to \infty} \lambda_k f(0)= 1,
  \end{align*} hence $f(0)\neq 0$ and $(\lambda_k)$  is not a null sequence either. \\
  Similarly, for  the function $h(z)=z$, there exists again  a sequence   $(\theta_j H^{n_j}f)$ in the projective orbit such that  $\theta_j H^{n_j}f \to h $ as $j\to \infty$, and
  \begin{align}
  \label{forone}
  \lim_{j\to \infty}\theta_j H^{n_j}f(0)= \lim_{j\to \infty} \theta_j f(0)= h(0)=0\nonumber\\
   \lim_{j\to \infty}\theta_j H^{n_j}f(1)= \lim_{j\to \infty} \theta_j f(1)= h(1)=1,
    \end{align} from which, since $f(0)\neq 0$, we deduce  $\theta_j \to 0$ as $j\to \infty$.  This and power boundedness in  Theorem~\ref{thm4} yield
    \begin{align*}
    \lim_{j\to \infty}\|\theta_j H^{n_j}f\|_{(p,\alpha, m)}\lesssim \| f\|_{(p,\alpha, m)} \lim_{j\to \infty}|\theta_j|= 0.
    \end{align*} Therefore, $\theta_j H^{n_j}f \to \textbf{0}$  which implies $\theta_j H^{n_j}f (1) \to 0$ as $j \to \infty$. This contradicts \eqref{forone}.
    \end{proof}

\section{ The Ritt's resolvent condition }
 A classical operator theoretic problem, for any given  bounded operator $T$ on a complex  Banach space $X$, is  to identify the relation between the size of the resolvent $(T-\lambda I)^{-1}$ when $\lambda$ is near to the spectrum of $T$ and the asymptotic behaviour of the orbits $\{T^nx: x\in X\}$. In this perspective, recall that the Ritt's condition \cite{Ritt}  for $T$ states that there exists a  positive constant $C$ such that
     \begin{align*}
     \|(T-\lambda I)^{-1}\| \leq  \frac{C}{|\lambda-1|}
     \end{align*} for each $\lambda \in \CC$  and  $|\lambda|>1$.  As an immediate consequence of   Theorem~\ref{thm4}, it turns out that  the  Hardy operator $H$ on  generalized Fock spaces  belongs to the class of operators satisfying such condition.
 \begin{proposition}
Let $1\leq p< \infty $. Then
\begin{enumerate}
\item the Hardy operator $H$   satisfies the Ritt's resolvent condition on $\mathcal{F}^p_{(\alpha, m)}$.
\item the differentiation operator   $D$   satisfies the Ritt's resolvent condition on $\mathcal{F}^p_{(\alpha, m)}$ if and only if it is power bounded and uniformly mean ergodic.
\end{enumerate}
\end{proposition}
\begin{proof}
(i) Nagy and Zemanek \cite{Nagy} proved  that a bounded  operator $T$ on a complex Banach space satisfies the Ritt's resolvent condition if and only if it is
power bounded and
\begin{align}
\label{Ritt}
\sup_{n\geq 1}n\|T^{n+1}-T^n\| <\infty.
\end{align} Thus, by  Theorem~\ref{thm4}, it is enough  to show that the operator $H$ satisfies condition \eqref{Ritt}.  To this goal, applying \eqref{simple}
\begin{align*}
 \|H^{n+1} f-H^n f\|_{(p,\alpha, m)}\leq \sum_{k=0}^\infty \frac{|a_k|\|z^k\|_{(p,\alpha, m)}}{(k+1)^n} \Big|\frac{1}{k+1}-1\Big|\\
    \leq  \| f\|_{(p,\alpha, m)} \sum_{k=0}^\infty \frac{k}{(k+1)^{n+1}}\simeq  \frac{\| f\|_{(p,\alpha, m)}}{n},
 \end{align*} and  the conclusion  easily follows.

(ii)  Assume first  that $D$ is power bounded and uniformly mean ergodic. Then by Theorem~\ref{thm1},  either $m<1$ or $m=1$ and $\alpha <1$. For $m<1$, arguing as in   \eqref{forresolvent}  there exist  $N\in \NN$ and $0 < \delta < 1$ such that
\begin{align}
\label{r1}
n\|D^{n+1}-D^n\|\leq n\|D^{n+1}\|+ n\|D^n\| \leq 2n\delta^{n} \ \ \text{for all} \ n \geq N.
\end{align}  Similarly, if  $m=1$ and $\alpha <1$, then \eqref{estdiff} implies
 \begin{align}
 \label{r2}
n\|D^{n+1}-D^n\|\leq n\|D^{n+1}\|+ n\|D^n\| \leq n\frac{(n+1)^{\frac{p}{2}}}{e^{(1-\alpha )(n+1)}}+n\frac{n^{\frac{p}{2}}}{e^{(1-\alpha )n}}.
\end{align}  Now we take the supremum with respect to $n\in \NN$ both in \eqref{r1} and \eqref{r2} to see that condition \eqref{Ritt} is  satisfied.

For  the other  implication, by  a result of Nagy and Zemanek \cite{Nagy}, it is enough to show that $D$ is uniformly mean ergodic.   We arrive at this conclusion if we show
that the Ritt's condition fails  when $m=1$ and $\alpha\geq  1$.  If $m=1$, then  using again the exponential functions, $e_{\beta}(z)=e^{\beta z}, \ |\beta| < \alpha$, we get $\overline{D(0,\alpha )} \subset \sigma (D)$.  Thus, then  the spectrum  $\sigma(D)$ contains the unit circle $\mathbb{T}$  whenever if  $\alpha\geq 1$. This is a contradiction since the spectrum of an operator which satisfies the Ritt's  resolvent condition contains only $1$ from the unit circle; see \cite{LY} and \cite[Theorem 4.5.4]{NE} for more details.
\end{proof}

\vspace{.2cm}

\subparagraph*{\textbf{Acknowledgement}}
The authors are very thankful to the referee for the careful reading of our paper and many suggestions which corrected and improved our manuscript.

Part of this  work  was done during the third-named author's stay at the Instituto Universitario de Matem\'atica Pura y Aplicada of the Universitat Polit\`ecnica de Val\`encia. He would like to thank Prof. Jos\'e Bonet, Prof. Alfred Peris and all other  members of the institute for their hospitality and kindness during his stay in Valencia, Spain.

\end{document}